\documentclass[reqno]{amsart}
\usepackage{cite}
\usepackage{fixltx2e}
\usepackage{amssymb,amsmath,amsfonts,amsthm}
\usepackage[utf8]{inputenc}
\usepackage{cmap}
\usepackage[T1]{fontenc}
\usepackage[all]{xy}
\usepackage{mathbbol}
\usepackage{graphicx}
\usepackage{xcolor}

\DeclareMathOperator{\Span}{Span}
\DeclareMathOperator{\Aut}{Aut}
\DeclareMathOperator{\LAut}{LAut}

\newtheorem{theorem}[subsection]{Theorem}
\newtheorem{proposition}[subsection]{Proposition}

\newtheorem{corollary}[subsection]{Corollary}
\theoremstyle{definition}
\newtheorem{definition}[subsection]{Definition}

\theoremstyle{remark}
\newtheorem{remark}[subsection]{Remark}

\begin{document}

\title[On Local Automorphisms of $\mathfrak{sl}_2$]{On Local Automorphisms of $\mathfrak{sl}_2$}
\author{T. Becker}
\address{[T. Becker] Sonoma State University, 94928 California, USA}
\email{terrisbecker@gmail.com}
\author{J. Escobar Salsedo}
\address{[J.Escobar] Sonoma State University, 94928 California, USA}
\email{escobar192011@gmail.com}
\author{C. Salas}
\address{[C. Salas] Sonoma State University, 94928 California, USA}
\email{crystalsalas44@gmail.com}
\author{R. Turdibaev}
\address{[R. Turdibaev] Inha University in Tashkent, Ziyolilar 9, 100170 Tashkent, Uzbekistan.}
\email{r.turdibaev@inha.uz}

\thanks{
The first three authors were supported by the National Science Foundation, grant number 1658672.}

\begin{abstract}
	We establish that the set of local automorphisms $\LAut(\mathfrak{sl}_2)$ is the group $\Aut^{\pm}(\mathfrak{sl}_2)$ of all automorphisms and anti-automorphisms. For $n\geq 3$ we prove that anti-automorphisms are local automorphisms of $\mathfrak{sl}_n$. 
\end{abstract}
\subjclass[2010]{}

\maketitle

\section*{Introduction}

The idea of a local automorphism dates back to 1988 \cite{Larson}, where for a given set of mappings $\mathcal{S}$ from a set $X$ into a set $Y$, Larson suggests to call a mapping $\theta$ to be interpolating $\mathcal{S}$  if for each $x\in X$ there is an element $S_x\in \mathcal{S}$, depending on $x$, with $\theta(x)=S_x(x)$. Set $X=Y=\mathcal{A}$ to be an algebra over a field $\mathbb{F}$ and $\mathcal{S}=\Aut(\mathcal{A})$ the set of automorphisms of $\mathcal{A}$, then a linear
endomorphism $\Delta$ of $\mathcal{A}$ is called \textit{a local automorphism} if it interpolates the set of automorphisms of $\mathcal{A}$. Local automorphisms have been 
introduced and studied for the algebra $B(X)$ of all bounded linear operators on infinite dimensional Banach space $X$ by Larson and Sourour \cite{LarsonSourour}.  It has been proven for the associative algebra of square matrices $\mathcal{M}_n(\mathbb{C})$ that any local automorphism is either an automorphism or an anti-automorphism \cite{LarsonSourour}. 

Investigation of local automorphisms continues for some subalgebras of $\mathcal{M}_n(\mathbb{C})$. In \cite{digraph_algebra} a description of the
local automorphisms of a finite-dimensional CSL algebra is given. A CSL algebra, or a digraph algebra, 
is an algebra which is spanned by a set of matrices which contains all diagonal matrix units $\{E_{ii}\}_{1\leq i \leq n}$ and is closed under multiplication. A local automorphism of a finite-dimensional CSL algebra is either an automorphism or an automorphism composed with a map that is the 
transpose map on a specific direct summand of the algebra and  the identity map 
on the complement to the summand. There is a discussion of local automorphisms of the 
algebra of niltriangular matrices $N(n,\mathcal{K})$ over an associative commutative ring $\mathcal{K}$ with identity in \cite{Elisova}. The authors of \cite{Elisova} provide a full description for $n=3$ and construct some non-trivial examples of local automorphisms for $n>3$. 

In our work we consider the simple Lie algebra $\mathfrak{sl}_n$ of traceless $n\times n$ matrices over a field of characteristic zero. With the formal investigations of local automorphisms of the associative matrix algebras in mind, our motivation is to investigate the local automorphisms of $\mathfrak{sl}_n$. 

For a finite-dimensional algebra $\mathfrak{g}$, the set of local automorphisms constitutes a group which we denote by $\LAut(\mathfrak{g})$. Clearly, $\Aut(\mathfrak{g})$ is a subgroup of $\LAut(\mathfrak{g})$, however it is not clear if it is a normal subgroup. 

We establish that any anti-automorphism is a local automorphism of $\mathfrak{sl}_n$. In fact, it is the composition of the matrix transposition and an automorphism of the algebra. We obtain a full description of a local automorphism of $\mathfrak{sl}_2$ -- 
it is either an automorphism or an anti-automorphism. Furthermore, $\Aut(\mathfrak{sl}_2)$ is a normal subgroup of $\LAut(\mathfrak{sl}_2)$ of index 2. 
We construct an example of a linear map that agrees with automorphisms on each basis element of $\mathfrak{sl}_n$ and is not a local automorphism. 


\section{Preliminaries}

\begin{definition} A vector space $\mathfrak{g}$ over $\mathbb{F}$ is called a Lie algebra if
	its multiplication (called  Lie 
	bracket and denoted by $(x,y) \mapsto [x,y]$) satisfies the identities:
	
	(1) $[x,x]=0$;
	
	(2) $[x,[y,z]]+[y,[z,x]]+[z,[x,y]]=0$.
	
\end{definition}

Every associative algebra $(\mathcal{A}, \ \cdot \ )$ turns into a Lie algebra $(\mathfrak{A},[-,-])$, where $\mathfrak{A}=\mathcal{A}$ and the Lie bracket is defined by the commutator $[a,b]=a\cdot b- b\cdot a$. This way the associative algebra $\mathcal{M}_n(\mathbb{F})$ of $n\times n$ square matrices over $\mathbb{F}$ turns into the Lie algebra $\mathfrak{gl}_n(\mathbb{F})$. Consider the subset of $\mathcal{M}_n(\mathbb{F})$ of traceless matrices. It is well-known that the product of traceless matrices is not necessarily traceless, hence it is not a subalgebra of $\mathcal{M}_n(\mathbb{F})$. This set is closed under 
the commutator and is a Lie subalgebra of $\mathfrak{gl}_n(\mathbb{F})$ denoted by $\mathfrak{sl}_n(\mathbb{F})$, which is the main focus of our work. 

In our work we consider a field $\mathbb{F}$ of characteristic zero and all definitions and results are restricted to this field. Hence, we omit $\mathbb{F}$ from now on.

A linear bijective map $\varphi:\mathfrak{g}\to \mathfrak{g}$ is called an automorphism (resp. an anti-automorphism), if it satisfies $\varphi([x,y])=[\varphi(x),\varphi(y)]$ (resp.  $\varphi([x,y])=[\varphi(y),\varphi(x)]$) for all $x,y\in \mathfrak{g}$. 	Analogous notions are defined in any algebra. 

The set of all automorphisms of $\mathfrak{g}$ constitutes a group with respect to composition, which is denoted by $\Aut(\mathfrak{g})$. The set of all anti-automorphisms of $\mathfrak{g}$ is denoted by $\Aut^{-}(\mathfrak{g})$. Due to skew-symmetry of the Lie bracket it is immediate that  $\varphi\in \Aut(\mathfrak{g})$ if and only if $-\varphi \in \Aut^{-}(\mathfrak{g})$. Note that the set $\Aut^{\pm}(\mathfrak{g})=\Aut(\mathfrak{g})\cup \Aut^{-}(\mathfrak{g})$ is a group called the \textit{signed automorphisms} group. Moreover, $\Aut(\mathfrak{g})$ is of index two in $\Aut^{\pm}(\mathfrak{g})$ and therefore is its normal subgroup \cite{JacobsonBA}.

A notion of a linear map which agrees with automorphisms at each point is given in the next definition.

\begin{definition}\label{definition_local}
	A linear map $\Delta:\mathfrak{g}\rightarrow\mathfrak{g}$ is a local automorphism if for any $X\in\mathfrak{g}$ there exists an automorphism $\varphi_X$ of $\mathfrak{g}$ such that $\Delta(X)=\varphi_X(X)$.
\end{definition}

It follows from Definition \ref{definition_local} that a local automorphism is an injective map.  The set of all local automorphisms is denoted by $\LAut(\mathfrak{g})$ and obviously there is an inclusion $\Aut(\mathfrak{g})\subseteq \LAut(\mathfrak{g})$. Moreover, it is straightforward that the composition of two local automorphisms is a local automorphism and $\LAut(\mathfrak{g})$ is a monoid. There is a statement \cite[Lemma 4]{Elisova} that $\LAut(\mathfrak{g})$ is a group  under the usual composition, which is not true since local automorphisms are generally speaking not surjective (see 
Theorem 3.11 of \cite{incidence}). However, the proof given in \cite{Elisova} works if $\mathfrak{g}$ is finite-dimensional. Indeed, then every local automorphism is surjective and given $\Delta \in\LAut(\mathfrak{g})$ we have $Y=\Delta(X)=\varphi_X(X)$ which implies $\Delta^{-1}(Y)=\varphi^{-1}_X(Y)$. Since $\varphi^{-1}_X\in \Aut(\mathfrak{g})$, the inverse of $\Delta$ also acts point-wise as an automorphism. Thus, there is the following  

\begin{proposition} \label{preliminary}	The set $\LAut(\mathfrak{g})$ constitutes a group with respect to composition if  $\mathfrak{g}$ is finite-dimensional.
\end{proposition}

For the associative algebra of all square matrices over the 
complex field, all local automorphisms are described by the following theorem. 

\begin{theorem}(\cite{LarsonSourour}) 
	A linear map $\alpha : \mathcal{M}_n(\mathbb{C}) \to \mathcal{M}_n(\mathbb{C})$ is a local automorphism iff $\alpha$ is an automorphism or an anti-automorphism, i.e., either $\alpha$ is 
	of the form $X\mapsto AXA^{-1}$  or $X \mapsto AX^T A^{-1}$ for a fixed $A\in \mathcal{M}_n(\mathbb{C})$.
\end{theorem}

In this work we study local automorphisms of $\mathfrak{sl}_n$. Let us denote the unit matrix with zero entries everywhere but the intersection of the 
$i$-th row and $j$-th column by $E_{ij}$.  Recall that the simple Lie algebra of $n\times n$ matrices of trace zero is generated as 
\begin{center}
	$\mathfrak{sl}_n=\Span \langle E_{ij}, h_1,\dots,h_{n-1} \mid 1\leq i\neq j\leq n \rangle,$
\end{center}
where $h_i=E_{ii}-E_{i+1,i+1}$.

Before presenting the main results, we introduce the following theorem which is used to prove our results.

\begin{theorem}\cite{JacobsonLie}\label{automorphisms} Over an algebraically closed field of characteristic 0, the group of automorphisms of the Lie algebra $\mathfrak{sl}_2$ is the set of mappings $X\mapsto A^{-1}XA$ and the group of automorphisms of the Lie algebra $\mathfrak{sl}_n (n\geq 3)$ is the set of mappings of the form $X\mapsto A^{-1}XA$ or $X\mapsto-A^{-1}X^TA$.
\end{theorem}

\begin{remark}\label{polynom}  
	In the case $n=2$ since there is only one type of automorphisms which is the 
	conjugation by a matrix, a local automorphism $\Delta$ of $\mathfrak{sl}_2$ sends a matrix $X$ to a similar matrix $\Delta(X)$. Since the characteristic polynomials of similar matrices are the same, we have $p_{X}(\lambda)=p_{\Delta(X)}(\lambda) $.
	The analogous statement for $n\geq 3$ holds only for the local automorphisms that act at each point as the 
	automorphism $X\mapsto A^{-1}XA$. 
\end{remark}

\section{Main Results}

\begin{proposition} Every anti-automorphism of $\mathfrak{sl}_n$ is a local automorphism.
\end{proposition}

\begin{proof} 
	Consider the transpose map $\Delta(X)=X^T$. It is a well-known fact that a square matrix is similar (conjugate) to its transpose \cite[Theorem 66]{MR2001037}. In this case, by Theorem \ref{automorphisms} we obtain that $\Delta$ is a local automorphism of $\mathfrak{sl}_n$. Furthermore, $\Delta$ is an anti-automorphism, and each anti-automorphism is of the form $\Delta \circ \varphi$, where $\varphi$ is an automorphism. Then $\Delta \circ \varphi$ is a local automorphism as the composition of two local automorphisms.	
\end{proof}

\begin{corollary}\label{signed_group_is_a_subgroup}
	The signed automorphism group $\Aut^\pm (\mathfrak{sl}_n)$ is a subgroup of  $\LAut(\mathfrak{sl}_n)$.
\end{corollary}

In general, we could not obtain a full description of local automorphisms of $\mathfrak{sl}_n$. However, in the case $n=2$ we achieve this goal.

Let us use the matrix representation of the elements of the Lie algebra $\mathfrak{sl}_2$: 
\begin{center}
	$e=\left(\begin{array}{cc}
	0&1 \\[1mm]
	0&0 \\[1mm]
	\end{array}\right), \ \ \ f=\left(\begin{array}{ll}
	0&0 \\[1mm]
	1&0 \\[1mm]
	\end{array}\right), \ \ \ h=\left(\begin{array}{cc}
	1&0 \\[1mm]
	0&-1 \\[1mm]
	\end{array}\right).$ 
\end{center}
Multiplication in this algebra is as follows:
\begin{center}
	$\begin{array}{lll}
	[e,f]=h,& [h,e]=2e, & [f,h]=2f.
	\end{array} $
\end{center}

First, let us establish the following result.

\begin{proposition}\label{local_h_fixed} A local automorphism $\Delta$ of $\mathfrak{sl}_2$ that fixes $h$ is either an automorphism with matrix $\left(\begin{array}{lll}
	\lambda & 0 & 0\\
	0 & \lambda^{-1} & 0\\
	0& 0 & 1 \\
	\end{array}\right)$
	or anti-automorphism with matrix $\left(\begin{array}{lll}
	0 & \mu & 0\\
	\mu^{-1} & 0 & 0\\
	0& 0 & 1 \\
	\end{array}\right)$ in the ordered basis $(e,f,h)$ with $\lambda, \mu \in \mathbb{F}^*$. 	
\end{proposition}

\begin{proof} Let $\Delta:\mathfrak{sl}_2\rightarrow\mathfrak{sl}_2$ be a local automorphism that fixes $h$. Then by the description of the automorphisms of $\mathfrak{sl}_2$ in Theorem \ref{automorphisms} we obtain that $\Delta(e)=T_e^{-1}eT_e$ and $\Delta(f)=T_f^{-1}fT_f$ for some invertible matrices $T_e$ and $T_f$. Simple manipulations show that 
	$$\Delta(e)=\frac{1}{|T_e|}(\alpha^2 e -\beta^2 f +\alpha\beta h) \text{ and } \Delta(f)=\frac{1}{|T_f|}(-\gamma^2 e +\delta^2 f -\gamma\delta h),
	$$
	where $(\beta \ \alpha)$ is the second row of $T_e$, $(\delta \ \gamma) $ is the first row of $T_f$, and $|A|$ is the determinant of $A$.
	
	Using the linearity of $\Delta$ we have $$\Delta(e+h)=\frac{1}{|T_e|}(\alpha^2 e -\beta^2 f 			+\alpha\beta h)+h=\left(\begin{array}{cc}
	\frac{\alpha\beta}{|T_e|}+1&\frac{\alpha^2}{|T_e|}\\			\frac{-\beta^2}{|T_e|} & -\frac{\alpha\beta}{|T_e|}-1 			\end{array}\right).$$
	
	The characteristic polynomial of the last matrix  is $-\displaystyle \frac{2\alpha\beta}{|T_e|}+x^2-1$, and by Remark \ref{polynom} is equal to the characteristic polynomial, $p_{e+h}(x)=x^2-1$, of $e+h$.
	This yields $\alpha\beta=0$. Similarly, $x^2-1=p_{f+h}(x)=p_{\Delta (f+h)}(x)=\displaystyle \frac{2\gamma\delta}{|T_f|}+x^2-1$ implies $\delta\gamma=0$. Moreover, $x^2-1=p_{e+f}(x)=p_{\Delta(e+f)}(x)=\displaystyle x^2-\frac{1}{|T_e||T_f|}(\alpha\delta-\gamma\beta)^2	$ implies $(\alpha\delta-\gamma\beta)^2=|T_e||T_f|$.
	Since $T_e$ and $T_f$ are invertible, $\alpha$ and $\beta$ cannot be zero at the same time. Similarly for $\alpha$ and $\gamma$, $\delta$ and $\gamma$, $\delta$ and $\beta$. This gives us the following cases:
	
	\medskip

	\textsc{Case 1.} $\alpha\neq 0$, $\beta=0$, $\delta\neq 0$, $\gamma=0$. This case leads to $|T_e||T_f|=(\alpha\delta)^2$.
	
	\medskip
	
	\textsc{Case 2.} $\alpha=0$, $\beta\neq 0$, $\delta=0$, $\gamma\neq 0$. In this case $|T_e||T_f|=(\gamma\beta)^2$.
	
	\noindent	The only possible result is two local automorphisms: 
	\begin{center}
		$\Delta_1(e)=\lambda e, \ \Delta_1(f)=\frac{1}{\lambda}f, \  \Delta_1(h)=h$
	\end{center} and
	\begin{center}
		$\Delta_2(e)=\mu f, \ \Delta_2(f)=\frac{1}{\mu}e, \ \Delta_2(h)=h,$
	\end{center}
	where $\lambda=\displaystyle \frac{\alpha^2}{|T_e|}$ and $\mu=\displaystyle \frac{-\beta^2}{|T_e|}$.  One can check that $\Delta_1$ is an automorphism and $\Delta_2$ is an anti-automorphism of $\mathfrak{sl}_2$.
\end{proof}

We now establish the full description of the local automorphisms of $\mathfrak{sl}_2$.

\begin{theorem} The group  $\LAut(\mathfrak{sl}_2)$ 
	coincides with $\Aut^\pm(\mathfrak{sl}_2)$. 
\end{theorem}
\begin{proof}
	Let $\Delta'$ be an arbitrary local automorphism of $\mathfrak{sl}_2$. For every $x\in \mathfrak{sl}_2$ we have $\Delta'(x)=\varphi_x(x)$ for some $\varphi_x\in \Aut(\mathfrak{sl}_2)$. Then  $\varphi^{-1}_h \circ \Delta'$ is a local automorphism as the composition of two local automorphisms and $(\varphi^{-1}_h \circ \Delta')(h)=\varphi^{-1}_h (\varphi_h(h))=h$. By Proposition \ref{local_h_fixed}  $\varphi^{-1}_h \circ \Delta'$ is either equal to the automorphism $\Delta_1$ or to 
	the anti-automorphism $\Delta_2$ defined in the proof of Proposition \ref{local_h_fixed}. Hence, $\Delta'$ is either an automorphism or an anti-automorphism, and  $\LAut(\mathfrak{sl}_2)\subseteq \Aut^\pm(\mathfrak{sl}_2)$. On the other hand, Corollary \ref{signed_group_is_a_subgroup} claims the converse inclusion. Thus, $\LAut(\mathfrak{sl}_2)=\Aut^\pm(\mathfrak{sl}_2)$. \end{proof}

\begin{corollary}
	Every local automorphism of $\mathfrak{sl}_2$ is either an automorphism $X\mapsto A^{-1}XA$ or an anti-automorphism $X\mapsto A^{-1}X^TA $ for any invertible $A$. 
\end{corollary}

The next proposition shows that for a linear map $\Delta:\mathfrak{sl}_n \to \mathfrak{sl}_n$  to be a local automorphism it is not enough to check that $\Delta$ acts as an automorphism on each basis element of $\mathfrak{sl}_n$.

\begin{proposition}\label{counterexample}
	Let $n\geq 3$ and $\alpha\in\mathbb{F}^*$. Define a linear map $\Delta_\alpha:\mathfrak{sl}_n\rightarrow
	\mathfrak{sl}_n$ by setting
	\begin{align*}
	&\Delta_\alpha(E_{1,n-1})=E_{n1}+\alpha E_{n,n-1}\\
	&\Delta_\alpha(E_{1n})=E_{1,n-1} \\
	&\Delta_\alpha(E_{n1})=E_{n,n-1}\\
	&\Delta_\alpha(E_{n,n-1})=E_{1n}
	\end{align*}
	and fixing all the other matrix units. Then $\Delta_\alpha$ agrees with automorphisms on each basis element of $\mathfrak{sl}_n$, but $\Delta_{\alpha} $ is not a local automorphism of $\mathfrak{sl}_n$. 
\end{proposition}

\begin{proof} Consider the matrices 	
	\begin{align*}
	T_1&= I+E_{1n}+E_{n-1,1} +(\alpha-1) E_{n-1,n-1}+E_{n,n-1}-E_{nn}, \\
	T_2&=I -E_{n-1,n-1}-E_{nn}+E_{n-1,n}+E_{n,n-1},\\
	T_3&=I-E_{11}+E_{1,n-1}+E_{n-1,1}-E_{n-1,n-1},\\
	T_4&=I-E_{11}+E_{1,n-1}-E_{n-1,n-1}+E_{n-1,n}+E_{n1}-E_{nn}.
	\end{align*}
	Note that they are invertible and the following equalities hold:
	\begin{align*}
	E_{1,n-1}T_1= & E_{1,n-1}+E_{11}+(\alpha-1)E_{1,n-1}=E_{11}+\alpha E_{1,n-1}\\
	=& (E_{n1}+E_{11}-E_{n1})+\alpha(E_{n,n-1}- E_{n,n-1}+E_{1,n-1})=T_1(E_{n1}+\alpha E_{n,n-1});\\
	E_{1n}T_2= & E_{1n}-E_{1n}+E_{1,n-1}= E_{1,n-1}= T_2E_{1,n-1}; \\
	E_{n1}T_3= & E_{n1}-E_{n1}+E_{n,n-1}=E_{n,n-1}=T_3E_{n,n-1};\\
	E_{n,n-1}T_4= & E_{n,n-1}-E_{n,n-1}+E_{nn}=E_{nn}=E_{1n}-E_{1n} +E_{nn}  = T_4E_{1n}.
	\end{align*}
	Therefore, 	
	\begin{align*}
	&\Delta_\alpha(E_{1,n-1})=E_{n1}+\alpha E_{n,n-1}=T^{-1}_1E_{1,n-1}T_1,\\
	&\Delta_\alpha(E_{1n})=E_{1,n-1}=T^{-1}_2E_{1n}T_2, \\
	&\Delta_\alpha(E_{n1})=E_{n,n-1}=T^{-1}_3E_{n1}T_3,\\
	&\Delta_\alpha(E_{n,n-1})=E_{1n}=T^{-1}_4E_{n,n-1}T_4.
	\end{align*}

	Since on the other basis elements of $\mathfrak{sl}_n$ the map $\Delta_{\alpha}$ acts as the identity, we obtain that $\Delta_\alpha$ is a linear transformation that acts as matrix conjugations on each basis element of $\mathfrak{sl}_n$. 
	
	Note that $\Delta_{\alpha}^2(E_{1,n-1})=E_{n,n-1}+\alpha E_{1n}$ is a matrix of rank $2$. Since $\LAut(\mathfrak{sl}_n)$ is a group, if $\Delta_{\alpha}$ is a local automorphism, then so is $\Delta_{\alpha}^2$. However, the automorphisms of $\mathfrak{sl}_n$ by Theorem \ref{automorphisms} preserve rank of matrices. Therefore, we have a contradiction and $\Delta_{\alpha}$ is not a local automorphism. 
\end{proof}

\section*{Acknowledgments}
 The first author was supported by the National Science Foundation, grant number 1658672. The second and third authors were supported by the National Science Foundation, grant number NSF HRD 1302873. Authors would like to thank Shavkat Ayupov and Bakhrom Omirov for suggesting to carry out this research and for their useful comments.

\end{document}